\documentclass{amsart}

\usepackage[T1]{fontenc}
\usepackage{enumerate, amsmath, amsfonts, amssymb, amsthm, mathrsfs, thmtools, wasysym, graphics, graphicx, xcolor, url, hyperref, hypcap, shuffle, xargs, multicol, overpic, pdflscape, multirow, hvfloat, minibox, accents, array, multido, xifthen, a4wide, ae, aecompl, blkarray, pifont, mathtools, etoolbox, dsfont}
\usepackage{marginnote}
\hypersetup{colorlinks=true, citecolor=darkblue, linkcolor=darkblue}
\usepackage[all]{xy}
\usepackage[bottom]{footmisc}
\usepackage{tikz}
\usepackage{tikz-qtree}
\usetikzlibrary{trees, decorations, decorations.markings, shapes, arrows, matrix, calc, fit, intersections, patterns, angles, snakes}
\usepackage[external]{forest}
\graphicspath{{figures/}{figures/nodes/}}
\makeatletter\def\input@path{{figures/}}\makeatother
\usepackage{caption}
\usepackage[noabbrev,capitalise]{cleveref}
\usepackage[export]{adjustbox}
\usepackage{ulem}

\newtheorem{theorem}{Theorem}%[section]
\newtheorem{corollary}[theorem]{Corollary}
\newtheorem{proposition}[theorem]{Proposition}
\newtheorem{lemma}[theorem]{Lemma}

\newtheorem*{theorem*}{Theorem}%[section]

\theoremstyle{definition}

\newtheorem{example}[theorem]{Example}
\newtheorem{remark}[theorem]{Remark}

\crefname{notation}{Notation}{Notations}
\crefname{problem}{Problem}{Problems}

\newcommand{\N}{\mathbb{N}} % naturals
\newcommand{\cC}{\mathcal{C}} % collection
\newcommand{\bb}[1]{\mathbb{#1}} % bb letters
\def\bQ{\mathbb Q}

\newcommand{\set}[2]{\left\{ #1 \;\middle|\; #2 \right\}} % set notation
\newcommand{\eqdef}{\mbox{\,\raisebox{0.2ex}{\scriptsize\ensuremath{\mathrm:}}\ensuremath{=}\,}} % :=
\newcommand{\qbinom}[2]{\genfrac{[}{]}{0pt}{}{#1}{#2}_q} % q-binomial

\DeclareMathOperator{\des}{des} % descents
\DeclareMathOperator{\asc}{asc} % ascents
\DeclareMathOperator{\agree}{agr} % agree
\DeclareMathOperator{\can}{can} % canopy
\let\val\relax
\DeclareMathOperator{\val}{val} % valleys
\DeclareMathOperator{\df}{df} % double falls
\DeclareMathOperator{\BB}{bb} % Bernardi-Bonichon map
\DeclareMathOperator{\intNodes}{inodes} % intermediate nodes

\newcommand{\ie}{\textit{i.e.}~} % id est
\newcommand{\eg}{\textit{e.g.}~} % exempli gratia
\definecolor{darkblue}{rgb}{0,0,0.7} % darkblue color
\definecolor{green}{RGB}{57,181,74} % darkblue color
\definecolor{violet}{RGB}{147,39,143} % darkblue color
\newcommand{\darkblue}{\color{darkblue}} % darkblue command
\newcommand{\defn}[1]{\textsl{\darkblue #1}} % emphasis of a definition
\newcommand{\OEIS}[1]{\cite[{\rm \href{http://oeis.org/#1}{\texttt{#1}}}]{OEIS}}
\usepackage{todonotes}

\newcommand{\meet}{\wedge} % meet
\newcommand{\join}{\vee} % join
\newcommand{\bigMeet}{\bigwedge} % meet
\newcommand{\bigJoin}{\bigvee} % join
\newcommandx{\projDown}[1][1={}]{\smash{\pi_\downarrow^{#1}}} % down projection map
\newcommandx{\projUp}[1][1={}]{\smash{\pi^\uparrow_{#1}}} % up projection map
\newcommand{\Tam}{\mathrm{Tam}} % Tamari lattice

\newcommand{\bbA}{\bb{A}}
\newcommand{\Ac}{A^{\circ}}

\makeatletter
\def\l@part{\@tocline{1}{8pt}{0pc}{}{}}
\def\l@section{\@tocline{1}{4pt}{0pc}{}{}}
\makeatother
\let\oldtocpart=\tocpart
\renewcommand{\tocpart}[2]{\sc\large\oldtocpart{#1}{#2}}
\let\oldtocsection=\tocsection
\renewcommand{\tocsection}[2]{\bf\oldtocsection{#1}{#2}}
\let\oldtocsubsubsection=\tocsubsubsection
\renewcommand{\tocsubsubsection}[2]{\quad\oldtocsubsubsection{#1}{#2}}

\title{Refined product formulas for Tamari intervals}

\thanks{AB \& FC were partially supported by the French grant DeRerumNatura (ANR-19-CE40-0018), and by the French--Austrian project EAGLES (ANR-22-CE91-0007 \& FWF I6130-N). VP was partially supported by the Spanish project PID2022-137283NB-C21 of MCIN/AEI/10.13039/501100011033 / FEDER, UE, by the Spanish--German project COMPOTE (AEI PCI2024-155081-2 \& DFG 541393733), by the Severo Ochoa and María de Maeztu Program for Centers and Units of Excellence in R\&D (CEX2020-001084-M), by the Departament de Recerca i Universitats de la Generalitat de Catalunya (2021 SGR 00697), and by the French--Austrian project PAGCAP (ANR-21-CE48-0020 \& FWF I 5788).}

\author{Alin Bostan}
\address[AB]{Inria, Sorbonne Université, LIP6, Paris, France}
\email{alin.bostan@inria.fr}
\urladdr{\url{https://mathexp.eu/bostan/}}

\author{Frédéric Chyzak}
\address[FC]{Inria, Palaiseau, France}
\email{frederic.chyzak@inria.fr}
\urladdr{\url{https://mathexp.eu/chyzak/}}

\author{Vincent Pilaud}
\address[VP]{Universitat de Barcelona \& Centre de Recerca Matemàtica, Barcelona}
\email{vincent.pilaud@ub.edu}
\urladdr{\url{https://www.ub.edu/comb/vincentpilaud/}}

\begin{document}

\begin{abstract}
We provide surprising product formulas for the $f$-vectors of the canonical complexes of the Tamari lattices and of the cellular diagonals of the associahedra.
\end{abstract}

\vspace*{-1.7cm}

\maketitle

%%%%%%%%%%%%%%%%%%%%%%%%%%%%%%%%%%%%%%

\section*{Introduction}

Consider the \defn{Tamari lattice}~$\Tam(n)$, whose elements are the binary trees with~$n$ nodes, and whose cover relations are given by right rotations~\cite{Tamari}.
An \defn{interval} (or \defn{relation}) is a pair~$S \le T$ in~$\Tam(n)$.
For a binary tree~$T$, we denote by~$\des(T)$ (resp.~by~$\asc(T)$) the number of binary trees covered by~$T$ (resp.~covering~$T$) in the Tamari lattice.
In other words, $\des(T)$ (resp.~$\asc(T)$) is the number of edges~$i \to j$ in~$T$ with~$i > j$ (resp.~with~$i < j$) when we orient the edges of~$T$ towards its root and label the nodes of~$T$ in inorder (meaning that we recursively label the left subtree of~$T$, then the root, and then recursively label the right subtree of~$T$).
The purpose of this paper is to prove the following two surprising formulas, whose first few values are gathered in \cref{table:fVectorCanonicalComplex,table:fVectorDiagonal}.

\begin{theorem}
\label{thm:fVectorCanonicalComplex}
For any~$n,k \in \N$, the number $a_{n,k}$ of intervals~$S \le T$ of the Tamari lattice~$\Tam(n)$ such that~$\des(S) + \asc(T) = k$ is given by
\[
a_{n,k} = \frac{2}{n(n+1)} \binom{n+1}{k+2} \binom{3n}{k}.
\]
\end{theorem}

\begin{theorem}
\label{thm:fVectorDiagonal}
For any~$n,k \in \N$, the sum $b_{n,k}$ of the binomial coefficients~$\binom{\des(S) + \asc(T)}{k}$ over all intervals~$S \le T$ of the Tamari lattice~$\Tam(n)$ is given by
\[
b_{n,k} = \sum_{\ell = k}^{n-1} a_{n,\ell} \binom{\ell}{k} = \frac{2}{(3n+1)(3n+2)} \binom{n-1}{k} \binom{4n+1-k}{n+1}.
\]
\end{theorem}

\begin{table}[t]
	\begin{tabular}{c|ccccccccc|c}
		$n \backslash k$ & $0$ & $1$ & $2$ & $3$ & $4$ & $5$ & $6$ & $7$ & $8$ & $\Sigma$\\
		\hline
		$1$ & $1$ &&&&&&&&& $1$ \\
		$2$ & $1$ & $2$ &&&&&&&& $3$ \\
		$3$ & $1$ & $6$ & $6$ &&&&&&& $13$ \\
		$4$ & $1$ & $12$ & $33$ & $22$ &&&&&& $68$ \\
		$5$ & $1$ & $20$ & $105$ & $182$ & $91$ &&&&& $399$ \\
		$6$ & $1$ & $30$ & $255$ & $816$ & $1020$ & $408$ &&&& $2530$ \\
		$7$ & $1$ & $42$ & $525$ & $2660$ & $5985$ & $5814$ & $1938$ &&& $16965$ \\
		$8$ & $1$ & $56$ & $966$ & $7084$ & $24794$ & $42504$ & $33649$ & $9614$ && $118668$ \\
		$9$ & $1$ & $72$ & $1638$ & $16380$ & $81900$ & $215280$ & $296010$ & $197340$ & $49335$ & $857956$
	\end{tabular}
	\caption{The first few values of $a_{n,k} = \frac{2}{n(n+1)} \binom{n+1}{k+2} \binom{3n}{k}$. Note that the first column is~$1$, the second column is~$n(n-1)$~\OEIS{A002378}, the last three diagonals are~\OEIS{A004321}, \OEIS{A006630}, and~\OEIS{A000139}, and the column sum is~\OEIS{A000260}. The $n$th row gives the $f$-vector of the canonical complex of the Tamari lattice~$\Tam(n)$.}
	\label{table:fVectorCanonicalComplex}
\end{table}

\begin{table}
	\begin{tabular}{c|ccccccccc}
		$n \backslash k$ & $0$ & $1$ & $2$ & $3$ & $4$ & $5$ & $6$ & $7$ & $8$ \\
		\hline
		$1$ & $1$ \\
		$2$ & $3$ & $2$ \\
		$3$ & $13$ & $18$ & $6$ \\
		$4$ & $68$ & $144$ & $99$ & $22$ \\
		$5$ & $399$ & $1140$ & $1197$ & $546$ & $91$ \\
		$6$ & $2530$ & $9108$ & $12903$ & $8976$ & $3060$ & $408$ \\
		$7$ & $16965$ & $73710$ & $131625$ & $123500$ & $64125$ & $17442$ & $1938$ \\
		$8$ & $118668$ & $604128$ & $1302651$ & $1540770$ & $1078539$ & $446292$ & $100947$ & $9614$ \\
		$9$ & $857956$ & $5008608$ & $12660648$ & $18086640$ & $15958800$ & $8898240$ & $3058770$ & $592020$ & $49335$
	\end{tabular}
	\caption{The first few values of $b_{n,k} = \frac{2}{(3n+1)(3n+2)} \binom{n-1}{k} \binom{4n+1-k}{n+1}$. Note that the first column is~\OEIS{A000260} while the diagonal is~\OEIS{A000139}. The $n$th row gives the $f$-vector of the cellular diagonal of the $(n-1)$-dimensional associahedron.}
	\label{table:fVectorDiagonal}
\end{table}

These formulas are of interest for two main reasons.
First, they count the faces of two complexes defined from the Tamari lattice and the associahedron: $(a_{n,k})_{0 \le k < n}$ is the $f$-vector of the \defn{canonical complex} of the Tamari lattice~\cite{Reading-arcDiagrams, Barnard, AlbertinPilaud}, while~$(b_{n,k})_{0 \le k < n}$ is the $f$-vector of the \defn{cellular diagonal of the associahedron}~\cite{SaneblidzeUmble-diagonals, MarklShnider,  Loday-diagonal, MasudaThomasTonksVallette, LaplanteAnfossi}.
Second, they have relevant specializations: on the one hand, $\sum_{\ell = 0}^{n-1} a_{n,\ell} = b_{n,0} = \frac{2}{(3n+1)(3n+2)} \binom{4n+1}{n+1}$ counts all Tamari intervals~\cite{Chapoton1} and \defn{rooted $3$-connected planar triangulations with $2n+2$ faces} (see~\cite{BernardiBonichon, FangFusyNadeau} for bijections, and~\OEIS{A000260} for more informations), and on the other hand, ${a_{n,n-1} = b_{n,n-1} = \frac{2}{n(n+1)} \binom{3n}{n-1}}$ counts \defn{synchronized Tamari intervals},  \defn{rooted non-separable planar maps with $n+1$ edges}, and \defn{$2$-stack sortable permutations of~$[n]$}, among others (see \cite{PrevilleRatelleViennot,FangPrevilleRatelle} and~\OEIS{A000139} for more informations).

The paper is organized as follows.
In \cref{sec:canonicalComplexDiagonalAssociahedron}, we present the connection between~$a_{n,k}$ and the canonical complex of the Tamari lattice, and the connection between~$b_{n,k}$ and the cellular diagonal of the associahedron.
In \cref{sec:analyticProof}, we show that the analytic approach of~\cite{Chapoton1, Chapoton2} can be adapted\footnote{We note that several analytic approaches (via creative telescoping, via binomial sums, or via a holonomic recurrence system) are actually possible, as was discussed in a preliminary longer version of this paper~\cite{BostanChyzakPilaud}.} to prove~\cref{thm:fVectorCanonicalComplex,thm:fVectorDiagonal}.
Finally, in \cref{sec:bijections}, we discuss bijective considerations\footnote{Further considerations, in particular on the (im)possibility to refine the formulas of~\cref{thm:fVectorCanonicalComplex,thm:fVectorDiagonal} by natural combinatorial parameters can be found in~\cite{BostanChyzakPilaud}.} and note that bijective proofs of~\cref{thm:fVectorCanonicalComplex,thm:fVectorDiagonal} can be derived from~\cite{FusyHumbert, FangFusyNadeau}.

%%%%%%%%%%%%%%%%%%%%%%%%%%%%%%%%%%%%%%

\section{Canonical complex of the Tamari lattice and diagonal of the associahedron}
\label{sec:canonicalComplexDiagonalAssociahedron}

In this section, we interpret the numbers~$a_{n,k}$ in terms of the canonical complex of the Tamari lattice (\cref{subsec:canonicalComplex}) and the numbers~$b_{n,k}$ in terms of the cellular diagonal of the associahedron (\cref{subsec:diagonalAssociahedron}).
These two interpretations are our motivations to study~$a_{n,k}$ and~$b_{n,k}$, but are not used beyond this section.
Rather than giving all details of the definitions of these objects, we thus prefer to refer to the original articles and only gather the essential material to make the connection.

\subsection{Canonical complex of the Tamari lattice}
\label{subsec:canonicalComplex}

A lattice~$(L, \le, \meet, \join)$ is \defn{join semidistributive} when $x \join y = x \join z$ implies~$x \join (y \meet z) = x \join y$.
Any~$x \in L$ then admits a \defn{canonical join representation}, which is a minimal irredundant representation~$x = \bigJoin J$ (for the order~$J \le J'$ if for any~$j \in J$, there is~$j' \in J'$ with~$j \le j'$).
The \defn{canonical join complex}~\cite{Reading-arcDiagrams, Barnard} of a join semidistributive lattice~$L$ is the simplicial complex of canonical join representations of the elements of~$L$.
Note that the dimension of the face of the canonical complex corresponding to an element~$x$ of~$L$ is the size of its canonical join representation, which is the number of elements covered by~$x$ in~$L$.
We define dually meet semidistributive lattices and their canonical meet complexes, and say that~$L$ is \defn{semidistributive} when it is both join and meet semidistributive.
The \defn{canonical complex}~\cite{AlbertinPilaud} of a semidistributive lattice~$L$ is the simplicial complex whose faces are~$J \sqcup M$ where~$x = \bigJoin J$ is the canonical join representation and~$y = \bigMeet M$ is the canonical meet representation for an interval~$x \le y$ in~$L$.
Note that the dimension of the face of the canonical complex corresponding to an interval~$x \le y$ is the number of elements covered by~$x$ in~$L$ plus the number of elements covering~$y$ in~$L$.
Observe also that the canonical complex is flag, meaning that it is the clique complex of its graph.

\begin{example}
\begin{figure}
	\centerline{\includegraphics[scale=.58]{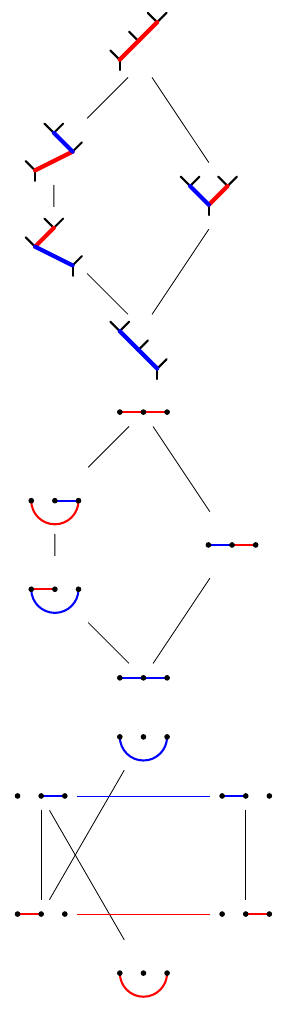}\quad\includegraphics[scale=.58]{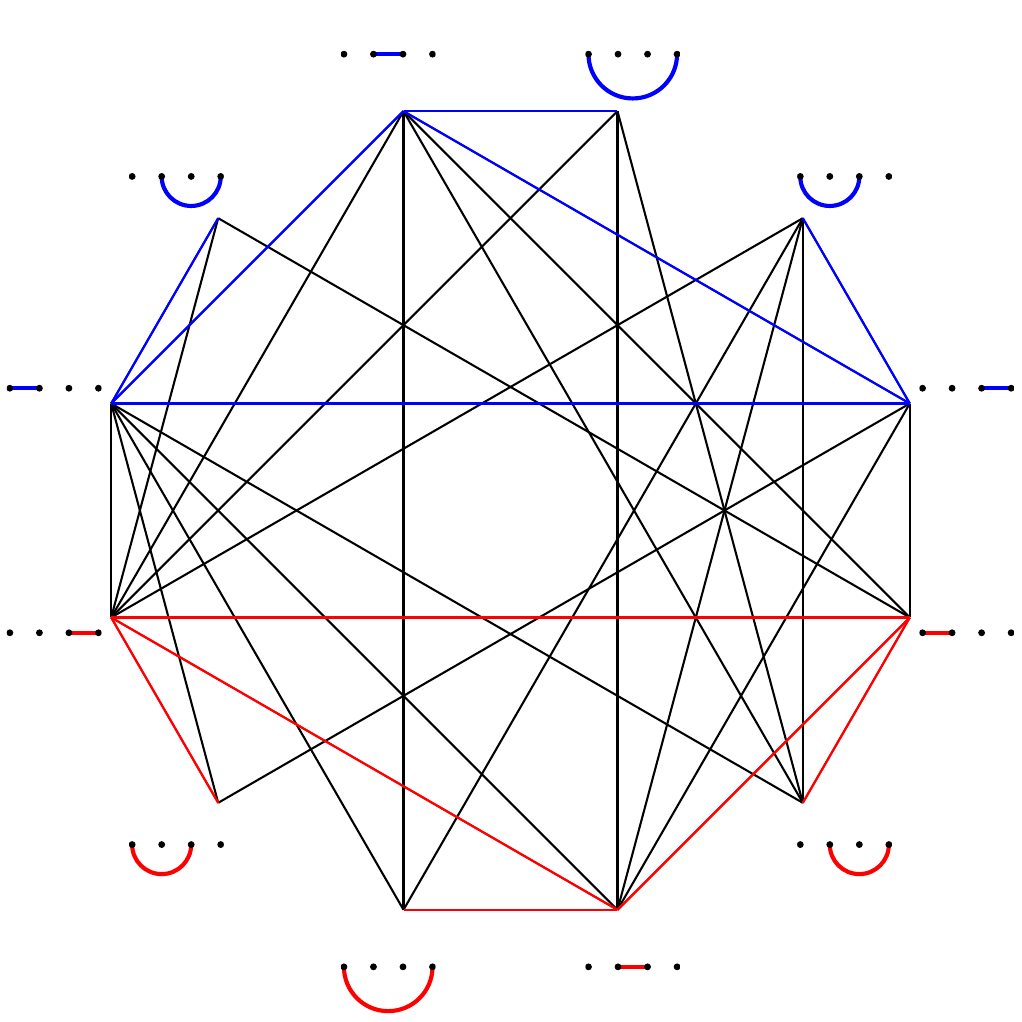}}
	\caption{The canonical complex of the Tamari lattice. Left: The Tamari lattice~$\Tam(2)$ seen on binary trees (top) and on semi-crossing arc bidiagrams (middle), and the canonical complex of~$\Tam(2)$ (bottom), with \mbox{$f$-vector}~$(1, 6, 6)$. Right: The canonical complex of~$\Tam(3)$, with \mbox{$f$-vector}~$(1,12, 33, 22)$.}
	\label{fig:canonicalComplex}
\end{figure}
The Tamari lattice is semidistributive.
Its join (resp.~meet) irreducible elements are given by binary trees~$T$ with~$\des(T) = 1$ (resp.~with~$\asc(T) = 1$), \ie with a single right (resp.~left) edge.
Such a tree is made by glueing two left (resp.~right) combs along a right (resp.~left) edge, and can thus be encoded by an arc.
The canonical join (resp.~meet) representation of a binary tree~$T$ is a non-crossing arc diagram with one arc for each right (resp.~left) edge of~$T$, which is also known as the non-crossing partition corresponding to~$T$.
Moreover, for a Tamari interval~$S \le T$, an arc~$j$ of the canonical join representation of~$S$ can cross an arc~$m$ of the canonical meet representation of~$T$ only if~$j$ passes from above to below~$m$.
The canonical complex of the Tamari lattice is thus called the semi-crossing complex.
This complex was extensively studied in~\cite{AlbertinPilaud} (note that the canonical complex of the Tamari lattice is just the restriction to down arcs of the canonical complex of the weak order, which was the one actually studied in~\cite{AlbertinPilaud}).
It is illustrated in \cref{fig:canonicalComplex}.
The top left picture shows the Tamari lattice where in each binary tree, the descents are colored red, and the ascents are colored blue.
The middle left picture is the translation on arcs, obtained by flattening each tree to the horizontal line.
The bottom left picture is the semi-crossing complex, thus the canonical complex of the Tamari lattice when~$n = 3$ (note that it has indeed $13$ faces: the empty set, $6$ vertices, and $6$ edges).
The right picture is the semi-crossing complex, thus the canonical complex of the Tamari lattice when~$n = 4$ (note that it has indeed $68$ faces: the empty set, $12$ vertices, $33$ edges, and $22$ triangles).
Note that we only draw the graphs of the canonical complexes, since they are flag simplicial complexes.
\end{example}

We are now ready to observe the connection between the numbers~$a_{n,k}$ of \cref{thm:fVectorCanonicalComplex} and the $f$-vector of the canonical complex of the Tamari lattice.
Recall that the \defn{$f$-vector} of a $d$-dimensional polytopal complex of~$\cC$ is the vector~$(f_0, f_1, \dots, f_d)$ where~$f_i$ denotes the number of $i$-dimensional faces of~$\cC$.

\begin{proposition}
\label{prop:fvectorCanonicalComplex}
The $f$-vector of the canonical complex of the Tamari lattice~$\Tam(n)$ on binary trees with~$n$ nodes is given by \((a_{n,k})_{0 \le k < n}.\)
\end{proposition}

\begin{proof}
The dimension of the face of the canonical complex of the Tamari lattice corresponding to an interval~$S \le T$ is the number of binary trees covered by~$S$ plus the number of binary trees covering~$T$, which is precisely~$\des(S) + \asc(T)$.
Hence, the number of $k$-dimensional faces of the canonical complex of~$\Tam(n)$ is given by~$a_{n,k}$.
\end{proof}

\subsection{Diagonal of the associahedron}
\label{subsec:diagonalAssociahedron}

The diagonal of a polytope~$P$ is the map~$\delta : P \to P \times P$ defined by~$x \mapsto (x,x)$.
A \defn{cellular approximation} of the diagonal of~$P$ (or just \defn{cellular diagonal} of~$P$ for short) is a map~$\tilde \delta : P \to P \times P$ homotopic to~$\delta$, which agrees with~$\delta$ on the vertices of~$P$, and whose image is a union of faces of~$P \times P$.
For a family of polytopes whose faces are products of polytopes in the family (like simplices, cubes, permutahedra or associahedra among others), some algebraic purposes additionally require the cellular diagonal to be compatible with the face structure.
Finding cellular diagonals in such families of polytopes is a difficult and important challenge at the crossroad of operad theory, homotopical algebra, combinatorics and discrete geometry, see \cite{SaneblidzeUmble-diagonals, MarklShnider, Loday-diagonal, MasudaThomasTonksVallette, LaplanteAnfossi} and the references therein.

Here, we focus on the associahedra.
Algebraic diagonals for the associahedra were found in~\cite{SaneblidzeUmble-diagonals} and later in~\cite{MarklShnider, Loday-diagonal}.
The first topological diagonal for the associahedra, as defined above, was given in~\cite{MasudaThomasTonksVallette} for the realizations of the associahedra of~\cite{Loday,ShniderSternberg}.
It recovers, at the cellular level, all the previous formulas~\cite{SaneblidzeUmble-comparingDiagonals, DelcroixOgerLaplanteAnfossiPilaudStoeckl}.
We simply denote by~$\Delta_d$ the cellular diagonal of the $d$-dimensional associahedron of~\cite{Loday,ShniderSternberg} constructed in~\cite{MasudaThomasTonksVallette}.
The faces of~$\Delta_d$ are given by the following description, called the \defn{magical formula}.

\begin{proposition}[{\cite[Thm.~2]{MasudaThomasTonksVallette}}]
\label{prop:magicalFormula}
The $k$-dimensional faces of the cellular diagonal~$\Delta_d$ correspond to the pairs~$(F,G)$ of faces of the associahedron with
\[
\dim(F) + \dim(G) = k
\qquad\text{and}\qquad
\max(F) \le \min(G)
\]
where~$\le$, $\max$ and~$\min$ refer to the order given by the Tamari lattice.
\end{proposition}

The method of~\cite{MasudaThomasTonksVallette}, fully developed in~\cite{LaplanteAnfossi} relies on the theory of fiber polytopes of~\cite{BilleraSturmfels}.
It enables to see the cellular diagonal of the associahedron as a polytopal complex refining the associahedron, a point of view we shall adopt in our figures for the rest of the paper.

\begin{example}
\begin{figure}
	\centerline{\includegraphics[scale=.5]{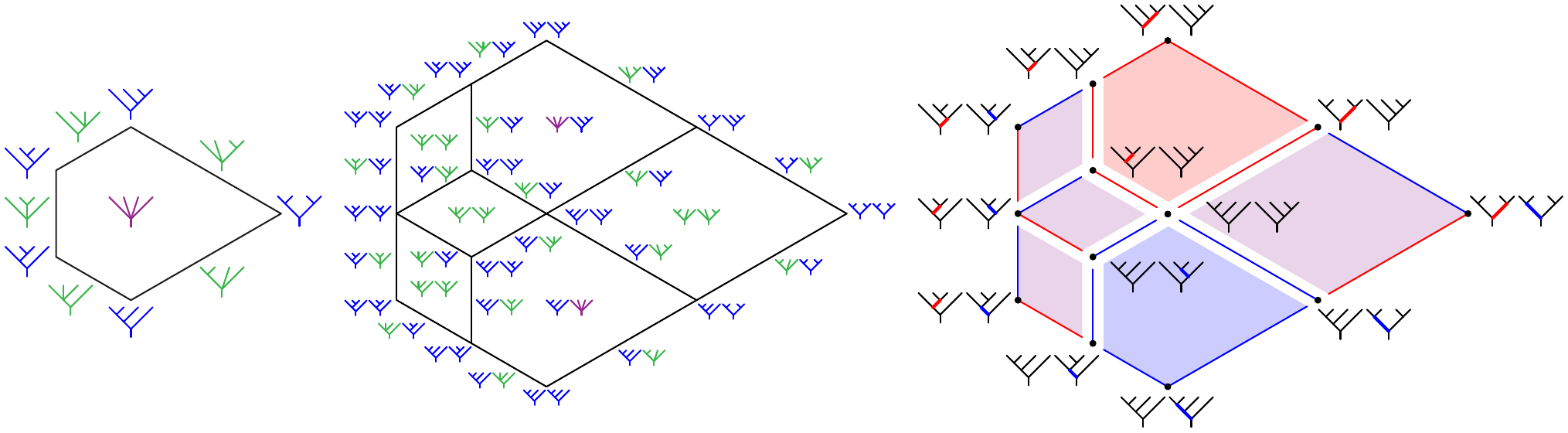}}
	\caption{Left: The $2$-dimensional associahedron with its faces labeled by Schr\"oder trees with $4$ leaves (in particular, its vertices correspond to binary trees). Middle: The cellular diagonal~$\Delta_2$ with its faces labeled by pairs of Schr\"oder trees given by the magical formula (in particular, its vertices correspond to Tamari intervals). Right: The decomposition of the cellular diagonal~$\Delta_2$ obtained by associating each face~$(F,G)$ to the Tamari interval~$\max(F) \le \min(G)$. The $f$-vector is~$(13,18,6)$.}
	\label{fig:diagonalAssociahedron}
\end{figure}
The cellular diagonal~$\Delta_2$ is illustrated in \cref{fig:diagonalAssociahedron}.
The left picture is the $2$-dimensional associahedron, with faces labeled by Schr\"oder trees (the colors depend on the dimension), and in particular with vertices labeled by binary trees.
The middle picture is the cellular diagonal~$\Delta_2$ seen as a polyhedral complex refining the $2$-dimensional associahedron, with faces labeled by pairs~$(F,G)$ of Schr\"oder trees, and in particular with vertices labeled by Tamari intervals.
The right picture is a decomposition of~$\Delta_2$, where each face~$(F,G)$ is associated to the Tamari interval~${\max(F) \le \min(G)}$.
In other words, the Tamari interval associated to a pair~$(F,G)$ of Schr\"oder trees is obtained by replacing each $p$-ary node of~$F$ (resp.~of~$G$) by a right (resp.~left) comb with $p$ leaves.
For each Tamari interval~$S \le T$, we have colored in red (resp.~blue) the edges of~$S$ (resp.~of~$T$) corresponding to descents of~$S$ (resp.~to ascents of~$T$).
\end{example}

We are now ready to observe the connection between the numbers~$b_{n,k}$ of \cref{thm:fVectorDiagonal} and the $f$-vector of the cellular diagonal of the $(n-1)$-dimensional associahedron.

\begin{proposition}
\label{prop:fVectorDiagonal}
The $f$-vector of the cellular diagonal~$\Delta_{n-1}$ of the $(n-1)$-dimensional associahedron is given by \((b_{n,k})_{0 \le k < n}.\)
\end{proposition}

\begin{proof}
For each binary tree~$T$, there are precisely~$\binom{\des(T)}{\ell}$ (resp.~$\binom{\asc(T)}{\ell}$) $\ell$-dimensional faces of the associahedron whose maximal (resp.~minimal) vertex is~$T$, because the associahedron is a simple polytope.
We thus directly derive from the magical formula of~\cref{prop:magicalFormula} that the number of $k$-dimensional faces of~$\Delta_{n-1}$ is
\[
\sum_{S \le T} \sum_{0 \le \ell \le k} \binom{\des(S)}{\ell} \binom{\asc(T)}{k-\ell} = \sum_{S \le T} \binom{\des(S) + \asc(T)}{k} = b_{n,k}.
\qedhere
\]
\end{proof}

\begin{remark}
This proof can also be interpreted on \cref{fig:diagonalAssociahedron}.
Namely, by attaching each face~$(F,G)$ to the Tamari interval~$\max(F) \le \min(G)$, we have partitioned the face poset of~$\Delta_{n-1}$ into boolean lattices based at its vertices.
As the boolean lattice attached to a Tamari interval~$S \le T$ has rank~$\des(S) + \asc(T)$, we obtain that the number of $k$-dimensional faces in this part of the face poset is~$\binom{\des(S) + \asc(T)}{k}$.
\end{remark}

%%%%%%%%%%%%%%%%%%%%%%%%%%%%%%%%%%%%%%

\section{Analytic proof}
\label{sec:analyticProof}

We now provide an analytic proof of~\cref{thm:fVectorCanonicalComplex,thm:fVectorDiagonal} following the approach of~\cite{Chapoton1, Chapoton2}.

%%%%%%%%%%%%%%%%%%%%%%%%%%%%%%%%%%%%%%

\subsection{Grafting decompositions}
\label{subsec:graftingDecompositions}

We first obtain a polynomial equation satisfied by the generating function $A(t,z) \eqdef \sum a_{n,k} t^n z^k$, that will be exploited in \cref{subsec:LagrangeInversion,subsec:binomialIdentity} to derive \cref{thm:fVectorCanonicalComplex,thm:fVectorDiagonal}.
Following the approach of~\cite{Chapoton1, Chapoton2}, we use a standard decomposition of Tamari intervals that naturally introduces an additional catalytic variable.

We denote by~$S / S'$ (resp.~by~$S' \backslash S$) the binary tree obtained by grafting the root of~$S$ on the leftmost (resp.~rightmost) leaf of~$S'$.
A grafting decomposition of~$S$ is an expression~${S = S_0 / S_1 / \dots / S_k}$ where~$S_i$ is a binary tree with at least a node.
In other words, a grafting decomposition of~$S$ is obtained by cutting some of the edges of~$S$ along the path from its root to its leftmost leaf.
See \cref{fig:graftingDecompositionTree}.
For a binary tree~$T$, we denote by~$n(T)$ the number of nodes of~$T$ and by~$\ell(T)$ the number of edges along the path from its root to its leftmost leaf (here, we only count edges between two nodes).
To fix the ideas, $n(Y) = 1$ and~$\ell(Y) = 0$ for the unique binary tree~$Y$ with a single node (and thus two leaves).
The following observations were made in~\cite[Sect.~3]{Chapoton1} and~\cite[Sect.~3.1]{Chapoton2}, and are illustrated in \cref{fig:graftingDecompositionInterval}.

\begin{lemma}[\cite{Chapoton1,Chapoton2}]
\label{lem:decomposition}
\phantom{.}
\begin{enumerate}[(i)]
\item Assume that~$S = S_0 / S_1 / \dots / S_k$ and~$T = T_0 / T_1 / \dots / T_k$ are such that~$n(S_i) = n(T_i)$ for all~$i \in [k]$. Then~$S \le T$ if and only if~$S_i \le T_i$ for all~$i \in [k]$.
\item If~$S \le T$, then we can write~$S = S_0 / S_1 / \dots / S_\ell$ and~$T = T_0 / T_1 / \dots / T_\ell$ where~$\ell = \ell(T)$ and $n(S_i) = n(T_i)$ for all~$i \in [\ell]$.
\label{item:decomposition}
\end{enumerate}
\end{lemma}

\begin{figure}[t]
	\centerline{\includegraphics[scale=1]{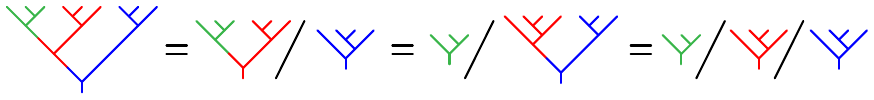}}
	\caption{All grafting decompositions of a binary tree.}
	\label{fig:graftingDecompositionTree}
\end{figure}

\begin{figure}
	\centerline{\includegraphics[scale=1]{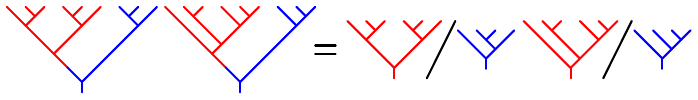}}
	\caption{A grafting decomposition of a Tamari interval.}
	\label{fig:graftingDecompositionInterval}
\end{figure}

Consider now the generating function
\[
\bbA(u,v,t,z) \eqdef \sum_{S \le T} u^{\ell(S)} v^{\ell(T)} t^{n(S)} z^{\des(S) + \asc(T)},
\]
where the sum ranges over all Tamari intervals (with arbitrary many nodes).
To simplify notations, we abbreviate~$A_u \eqdef A_u(t,z) \eqdef \bbA(u,1,t,z)$ and~$\Ac_u \eqdef \Ac_u(t,z) \eqdef \bbA(u,0,t,z)$.
Note that
\[
A_1(t,z) \eqdef \bbA(1,1,t,z) = A(t,z).
\]
Observe also that~$\Ac_u(t,z)$ is the generating function of indecomposable Tamari intervals, \ie of Tamari intervals~$S \le T$ where~$\ell(T) = 0$ so that the decomposition of \cref{lem:decomposition}\,\eqref{item:decomposition} is trivial.
\cref{lem:decomposition} leads to the following functional equation connecting~$A_u$ and~$A_1$.

\begin{proposition}
\label{prop:quadraticEquationA}
The generating functions~$A_u \eqdef \bbA(u,1,t,z)$ and~$A_1 \eqdef \bbA(1,1,t,z)$ satisfy the quadratic functional equation
\[
(u-1) A_u = t \big( u-1 + u (u+z-1) A_u - z A_1 \big) \big( 1 + uz A_u \big).
\]
\end{proposition}

\begin{proof}
This statement could be directly deduced by substituting~$x = 1$ and $y = \bar y = z$ in the equation given in~\cite[Prop.~1]{Chapoton2}.
For completeness, we prefer to transpose the proof as a simpler version of the proof of~\cite[Prop.~1]{Chapoton2} is sufficient for our purpose.

By definition, any Tamari interval~$S \le T$ is either indecomposable or can be decomposed as ${S = S' / S''}$ and~$T = T' / T''$ for an indecomposable Tamari interval~$S' \le T'$ and an arbitrary Tamari interval~$S'' \le T''$.
Since~$\ell(S) = \ell(S') + \ell(S'')+1$, $n(S) = n(S') + n(S'')$, $\des(S) = \des(S') + \des(S'')$, and~$\asc(T) = \asc(T') + \asc(T'') + 1$, we obtain
\begin{equation}
\label{eq:C1}
A_u = \Ac_u + u z \Ac_u A_u.
\end{equation}
Now from any Tamari interval~$(S,T)$ where~$S = S_0 / S_1 / \dots / S_{\ell(S)}$, we can construct $\ell(S)+2$ indecomposable Tamari intervals~$(S_k',T')$ for~$0 \le k \le \ell(S)+1$, where
\[
S_k' = \big( S_0 / \dots / S_{k-1} \big) / Y \backslash \big( S_k / \dots / S_{\ell(S)} \big)
\qquad\text{and}\qquad
T' = Y \backslash T
\]
(recall that $Y$ denotes the unique binary tree with a single node).
\begin{figure}[t]
	\centerline{
		\begin{tabular}{c@{\quad}c@{\quad}c@{\quad}c}
			\includegraphics[scale=1]{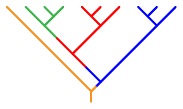} &
			\includegraphics[scale=1]{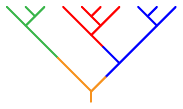} &
			\includegraphics[scale=1]{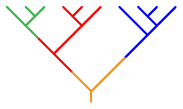} &
			\includegraphics[scale=1]{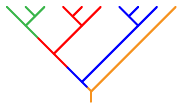} \\
			$S_0' = Y \backslash (S_0 / S_1 / S_2)$ &
			$S_1' = S_0 / Y \backslash (S_1 / S_2)$ &
			$S_2' = (S_0 / S_1) / Y \backslash S_2$ &
			$S_3' = (S_0 / S_1 / S_2) / Y$
		\end{tabular}
	}
	\caption{The binary trees~$S_k'$ for~$0 \le k \le 3$ obtained from the binary tree~$S$ of \cref{fig:graftingDecompositionTree} in the proof of \cref{prop:quadraticEquationA}.}
	\label{fig:proofGraftingDecompositions}
\end{figure}
See \cref{fig:proofGraftingDecompositions}.
For the extreme values of~$k$, we have~$S_0' = Y \backslash S$ and~$S_{\ell(S)}' = S / Y$.
Moreover, any indecomposable Tamari interval~$(S',T')$ with~$n(S')=n(T') > 1$ is obtained in a single way by this procedure.
Since  $\ell(S_k') = k$, ${n(S') = n(S) + 1}$, $\des(S') = \des(S)+1$ when~$k \le \ell(S)$ while ${\des(S_{\ell(S)+1}') = \des(S)}$, and $\asc(T') = \asc(T)$, we obtain
\begin{equation}
\label{eq:C2}
\Ac_u = t \Big( 1 + z \frac{u A_u - A_1}{u-1} + u A_u \Big).
\end{equation}
Combining \cref{eq:C1,eq:C2}, we finally get
\[
\label{eq:C3}
A_u = t \Big( 1 + z \frac{u A_u - A_1}{u-1} + u A_u \Big) \big( 1 + u z A_u \big),
\]
which rewrites as
\[
(u-1) A_u = t \big( u-1 + u (u+z-1) A_u - z A_1 \big) \big( 1 + u z A_u \big).
\qedhere
\]
\end{proof}

We are now ready to derive our functional equation on~$A$ using the quadratic method~\cite{GouldenJackson}.

\begin{proposition}
\label{prop:polynomialEquationA}
The generating function~$A = A(t,z)$ is a root of the polynomial~$P(t,z,X)$ of~$\bQ[t,z,X]$ given by
\begin{gather*}
t^3 z^6 X^4 \\
{} + t^2 z^4 (t z^2 + 6 t z - 3 t + 3) X^3 \\
{} + t z^2 (6 t^2 z^3 + 9 t^2 z^2 - 12 t^2 z + 2 t z^2 + 3 t^2 - 6 t z + 21 t + 3) X^2 \\
{} + (12 t^3 z^4 - 4 t^3 z^3 - 9 t^3 z^2 - 10 t^2 z^3 + 6 t^3 z + 26 t^2 z^2 - t^3 + 6 t^2 z + t z^2 + 3 t^2 - 12 t z - 3 t + 1) X \\
{} + t (8 t^2 z^3 - 12 t^2 z^2 + 6 t^2 z - t z^2 - t^2 + 10 t z + 2 t - 1).
\end{gather*}
\end{proposition}

\begin{proof}
We simply apply the quadratic method~\cite{GouldenJackson}.
The quadratic equation of \cref{prop:quadraticEquationA} can be rewritten as
\(
\alpha A_u^2 + \beta A_u + \gamma = 0,
\)
where
\[
\alpha = t u^2 z (u+z-1),
\quad
\beta = t u (u+z-1) + t u z (u-1) - t u z^2 A_1 - u + 1,
\quad
\gamma = t (u-1) - t z A_1.
\]
The discriminant~$\Delta \eqdef \beta^2 - 4 \alpha \gamma$
must have multiple roots, which implies that its own discriminant in~$u$ vanishes.
Removing clearly non-vanishing factors, this leads to the equation of the statement.
Note that~$\Delta$ having only degree~$4$ in~$u$, the formula for the discriminant could be worked out by hand.
\end{proof}

\begin{remark}
\label{rem:z=0}
When specialized at~$z = 0$, \cref{prop:polynomialEquationA} shows that~$A(t,0)$ is a root of the polynomial
\[
P(t,0,X) = - (t - 1)^3 X - t (t - 1)^2
\]
which recovers the fact that~$A(t,0) = t / (1 - t) = t + t^2 + t^3 + \cdots$.
\end{remark}

\begin{remark}
\label{rem:z=1}
When specialized at~$z = 1$,  \cref{prop:polynomialEquationA} shows that $A(t,1)$ is a root of the polynomial
\[
P(t,1,X) = t^3 X^4 + t^2 (4 t + 3) X^3 + t (6 t^2 + 17 t+ 3) X^2 + (4 t^3 + 25 t^2 - 14 t + 1) X + t^3 + 11 t^2 - t.
\]
This is the classical functional equation for the generating function of Tamari intervals (see \eg~\cite[Eq.~(5)]{Chapoton1}).
The curve defined by $P(t,1,X)$ has genus zero and admits the rational parametrization
\begin{equation}\label{eq:para1}
 t = \frac{s}{(s+1)^4}, \quad X = s - s^2 - s^3.
\end{equation}
As a consequence, the unique root $A = A(t,1)= t +3 t^{2}+13 t^{3}+68 t^{4}+399 t^{5}+2530 t^{6}+\cdots$ in~$\bQ[[t]]$ of the
polynomial $P(t,1,X)$ can be written as
\begin{equation}\label{eq:AofS1}
A = S - S^2 - S^3,
\end{equation}
where
\(
S = t +4 t^{2}+22 t^{3}+140 t^{4}+ \cdots
\)
is the unique solution in $\bQ[[t]]$ of
\[\label{eq:Soft1}
t = \frac{S}{(S + 1)^4}.
\]
From this equation, the coefficients of~$S$, $S^2$ and~$S^3$ can be computed via Lagrange inversion.
More precisely, for $r \ge 1$, Lagrange inversion gives
\[
[t^n] S^r = \frac{1}{n} [s^{n-1}] \, r  s^{r-1} \phi(s)^n = \frac{r}{n} [s^{n-r}] \phi(s)^n,
\]
where~$\phi(s) \eqdef (s+1)^4$. Since
\[
[s^a] \phi(s)^n = [s^a] (s+1)^{4n} = \binom{4n}{a},
\]
we obtain that, for~$r \in \{1,2,3\}$,
\[
[t^n] S^r = \frac{r}{n} [s^{n-r}] \phi(s)^n = \frac{r}{n} \binom{4n}{n-r}.
\]
Hence, \cref{eq:AofS1} implies that
\[
[t^n] A = [t^n] S - [t^n] S^2 - [t^n] S^3
\]
is given by
\[
\frac{1}{n} \left( \binom{4n}{n-1} - 2\binom{4n}{n-2} - 3\binom{4n}{n-3}  \right) = \frac{2}{(3n+1)(3n+2)} \binom{4n+1}{n+1},
\]
as proved in~\cite[Thm.~2.1]{Chapoton1}.
\end{remark}

%%%%%%%%%%%%%%%%%%%%%%%%%%%%%%%%%%%%%%

\subsection{\cref{thm:fVectorCanonicalComplex} by Lagrange inversion}
%\subsection{From the functional equation to the product formula: Lagrange inversion}
\label{subsec:LagrangeInversion}

We will now mimic the approach in \cref{rem:z=1}, and
extract the coefficients of~$A(t,z)$ to obtain \cref{thm:fVectorCanonicalComplex}.
The starting point is that the curve in $t, X$ defined by the
polynomial $P(t,z,X) \in \bQ(z)[t,X]$ from
\cref{prop:polynomialEquationA} still has genus zero and
admits the following rational parametrization:
\begin{equation}\label{eq:para}
 t = \frac{s}{(s+1) (sz+1)^3}, \quad X = s - z s^2 - z s^3,
\end{equation}
which lifts the parametrization~\eqref{eq:para1}.
As a consequence, the unique root $A$ in $\bQ[[t,z]]$ of the
polynomial $P(t,z,X)$ can be written
\begin{equation}\label{eq:AofS}
A = S - z S^2 - z S^3,
\end{equation}
where
\(
S= t +\left(3 z +1\right) t^{2}+\left(12 z^{2}+9 z +1\right) t^{3}+\cdots \)
is the unique solution in $\bQ[z][[t]]$ of
\begin{equation}\label{eq:Soft}
t = \frac{S}{(S+1) (Sz+1)^3}.
\end{equation}

There exist infinitely many rational parametrizations of $P$,
but the one in~\cref{eq:para} has a double advantage:
on the one hand, \cref{eq:Soft} is under a form amenable to Lagrange
inversion, and therefore allows to express the coefficient of $z^k
t^n$ in $S$ and in its powers;
on the other hand, the simple form of \cref{eq:AofS} allows to
easily extract the coefficient of $t^n z^k$ in $A$ as a sum of
similar coefficients of $S$, $S^2$ and $S^3$.
Putting together~\cref{eq:AofS,eq:Soft} enables us to express
the coefficient of $t^n z^k$ in $A$ as a binomial sum.
Let us give a few more details.

For $r \ge 1$ Lagrange inversion gives
\[
[t^n z^k] S^r = \frac{1}{n} [s^{n-1} z^k] r  s^{r-1} \phi(s)^n = \frac{r}{n} [s^{n-r} z^k] \phi(s)^n,
\]
where~$\phi(s) \eqdef (s+1) (sz+1)^3$.
We have that
\[
[s^a] \phi(s)^n = [s^a] (s+1)^n  (sz+1)^{3n} = \sum_{i+j=a} \binom{n}{i} \binom{3n}{j} z^j,
\]
and therefore
\[
[s^a z^k] \phi(s)^n =  \binom{n}{a-k}  \binom{3n}{k}.
\]
It follows that, for~$r \in \{1,2,3\}$,
\[
[t^n z^k] S^r = \frac{r}{n} [s^{n-r} z^k] \phi(s)^n = \frac{r}{n} \binom{n}{n-r-k}  \binom{3n}{k} = \frac{r}{n} \binom{n}{k+r}  \binom{3n}{k},
\]
Hence, \cref{eq:AofS} implies that
\[
a_{n,k} = [t^n z^k] A = [t^n z^k] S - [t^n z^{k-1}] S^2 - [t^n z^{k-1}] S^3
\]
is given by
\[
\frac{1}{n} \left( \binom{n}{k+1} \binom{3n}{k} - 2 \binom{n}{k+1} \binom{3n}{k-1} - 3 \binom{n}{k+2} \binom{3n}{k-1} \right) = \frac{2}{n(n+1)}  \binom{3n}{k}  \binom{n+1}{k+2},
\]
which proves \cref{thm:fVectorCanonicalComplex}.

\subsection{\cref{thm:fVectorDiagonal} by a binomial identity}
\label{subsec:binomialIdentity}

We now simply derive~\cref{thm:fVectorDiagonal} from \cref{thm:fVectorCanonicalComplex}, which amounts to checking the following binomial identity.

\begin{proposition}
\label{prop:binomialIdentity}
For any~$n,k \in \N$,
\[
\sum_{\ell = k}^{n-1} \frac{2}{n(n+1)} \binom{n+1}{\ell+2} \binom{3n}{\ell} \binom{\ell}{k} = \frac{2}{(3n+1)(3n+2)} \binom{n-1}{k} \binom{4n+1-k}{n+1}.
\]
\end{proposition}

We shall actually prove the following generalization.

\begin{proposition}
\label{prop:generalizedBinomialIdentity}
For any~$n,k,r \in \N$,
\[
\sum_{\ell =k}^{n-1}  \binom{n+1}{\ell+2} \binom{r}{\ell} \binom{\ell}{k} = \frac{n(n+1)}{(r+1)(r+2)} \binom{n-1}{k} \binom{r+n+1-k}{n+1}.
\]
\end{proposition}

\begin{proof}
Using the identity
\[
\binom{r}{\ell} \binom{\ell}{k} = \binom{r}{k} \binom{r - k}{r - \ell}
\]
this amounts to showing that
\[
 \binom{r}{k} \sum_{\ell \geq 0}  \binom{n+1}{\ell+2} \binom{r - k}{r - \ell} = \frac{n(n+1)}{(r+1)(r+2)} \binom{n-1}{k} \binom{r+n+1-k}{n+1}.
\]
This is in turn equivalent to
\[
 \sum_{\ell \geq 0}  \binom{n+1}{\ell+2} \binom{r - k}{r - \ell}
 =
  \binom{r + n + 1 - k}{r+2},
\]
which is a particular case of the classical Chu--Vandermonde identity.
\end{proof}

%%%%%%%%%%%%%%%%%%%%%%%%%%%%%%%%%%%%%%

\section{Bijections}
\label{sec:bijections}

In this section, we present some bijective considerations on \cref{thm:fVectorCanonicalComplex,thm:fVectorDiagonal}.
We first describe some statistics equivalent to $\des(S)$ and~$\asc(T)$ (\cref{subsec:equivalentStatistics}), expressed in terms of canopy agreements in binary trees (\cref{subsubsec:canopy}), of valleys and double falls in Dyck paths (\cref{subsubsec:DyckPaths}), and of internal degree of Schnyder woods in planar triangulations (\cref{subsubsec:triangulations}).
We then use bijective results of~\cite{FusyHumbert} to provide a more bijective proof of \cref{thm:fVectorCanonicalComplex} (\cref{subsec:triangulations}).

\subsection{Equivalent statistics}
\label{subsec:equivalentStatistics}

Transporting the ascent and descent statistics, we can interpret the formulas of \cref{thm:fVectorCanonicalComplex,thm:fVectorDiagonal} on other combinatorial families encoding Tamari intervals.
Here, we provide three alternative interpretations which seem particularly relevant to us.

\subsubsection{Canopy agreements}
\label{subsubsec:canopy}

Recall that the \defn{canopy} of a binary tree~$T$ with~$n$ nodes is the vector~$\can(T)$ of~$\{-,+\}^{n-1}$ whose $j$th coordinate is~$-$ if and only if the following equivalent conditions are satisfied:
\begin{enumerate}[(i)]
\item the $(j+1)$st leaf of~$T$ is a right leaf,
\item there is an oriented path joining its $j$th node to its $(j+1)$st node, \label{item:path}
\item the $j$th node of $T$ has an empty right subtree,
\item the $(j+1)$st node of~$T$ has a non-empty left subtree, \label{item:subtree}
\item the cone corresponding to~$T$ is located in the halfspace~$x_j \le x_{j+1}$.
\end{enumerate}
(In all these conditions, recall that~$T$ is labeled in inorder and oriented towards its root.)
We need the following three immediate observations, illustrated in \cref{fig:ascentsDescentsCanopyDyckPathsSchnyderWoods,fig:diagonalAssociahedronCanopyDyckPathsSchnyderWoods}.

\begin{figure}[b]
	\centerline{\includegraphics[scale=1.2]{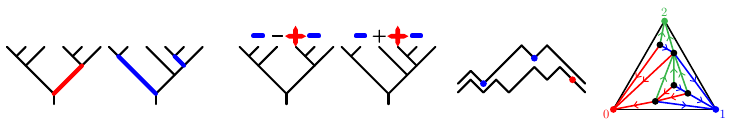}}
	\caption{Connections between equivalent statistics. The descents of~$S$ (resp.~ascents of~$T$) on the left correspond to the positions where the canopies of~$S$ and~$T$ are both positive (resp.~negative) in the middle left, to the double falls of~$\pi(S)$ (resp.~the valleys of~$\pi(T)$) in the middle right, and to the intermediate nodes of the tree~$T_0$ (resp.~$T_1$) on the right.}
	\label{fig:ascentsDescentsCanopyDyckPathsSchnyderWoods}
\end{figure}

\begin{lemma}
\label{lem:canopy}
For any binary trees~$S$ and~$T$,
\begin{enumerate}[(i)]
\item the number of $-$ (resp.~$+$) entries in the canopy of~$T$ is given by~$\asc(T)$ (resp.~by~$\des(T)$),
\item if~$S \le T$ in Tamari order, then the canopy of~$S$ is componentwise smaller than the canopy of~$T$ for the natural order~$- \le +$,
\item if~$S \le T$, then the number of positions where the entries of the canopies of both~$S$ and~$T$ are~$-$ (resp.~$+$) is given by~$\asc(T)$ (resp.~by~$\des(S)$).
\end{enumerate}
\end{lemma}

\begin{proof}
\begin{enumerate}[(i)]
\item By the characterization~\eqref{item:subtree} of the canopy above, $\can(T)_j = {-}$ if and only if there is an edge~$i \to j+1$ for some~$i \le j$, which thus defines an ascent of~$T$. Hence, the number of~$-$ entries in~$\can(T)$ is~$\asc(T)$. By symmetry, the number of $+$ entries in~$\can(T)$ is~$\des(T)$
\item It is sufficient to prove (ii) for a cover relation in the Tamari order. If the edge $i \to j$ with~$i < j$ is rotated, then the canopy is unchanged, except maybe its $i$th entry, which changes from~$-$ to~$+$ when~$j = i+1$. An alternative global argument is to observe that if~$S \le T$, then any linear extension of~$S$ is smaller than any linear extension of~$T$, so that there cannot be both oriented paths from~$i+1$ to $i$ in~$S$ and from $i$ to~$i+1$ in~$T$, and to use the characterization~\eqref{item:path} of the canopy above.
\item We have $\can(S)_j = \can(T)_j = {-}$ if and only if $\can(T)_j = {-}$ (by (ii)), so that the number of such positions is~$\asc(T)$ by~(i). By symmetry, the number of positions~$j$ with $\can(S)_j = \can(T)_j = {+}$ is~$\des(S)$.
\qedhere
\end{enumerate}
\end{proof}

\begin{figure}
	\centerline{\includegraphics[scale=.5]{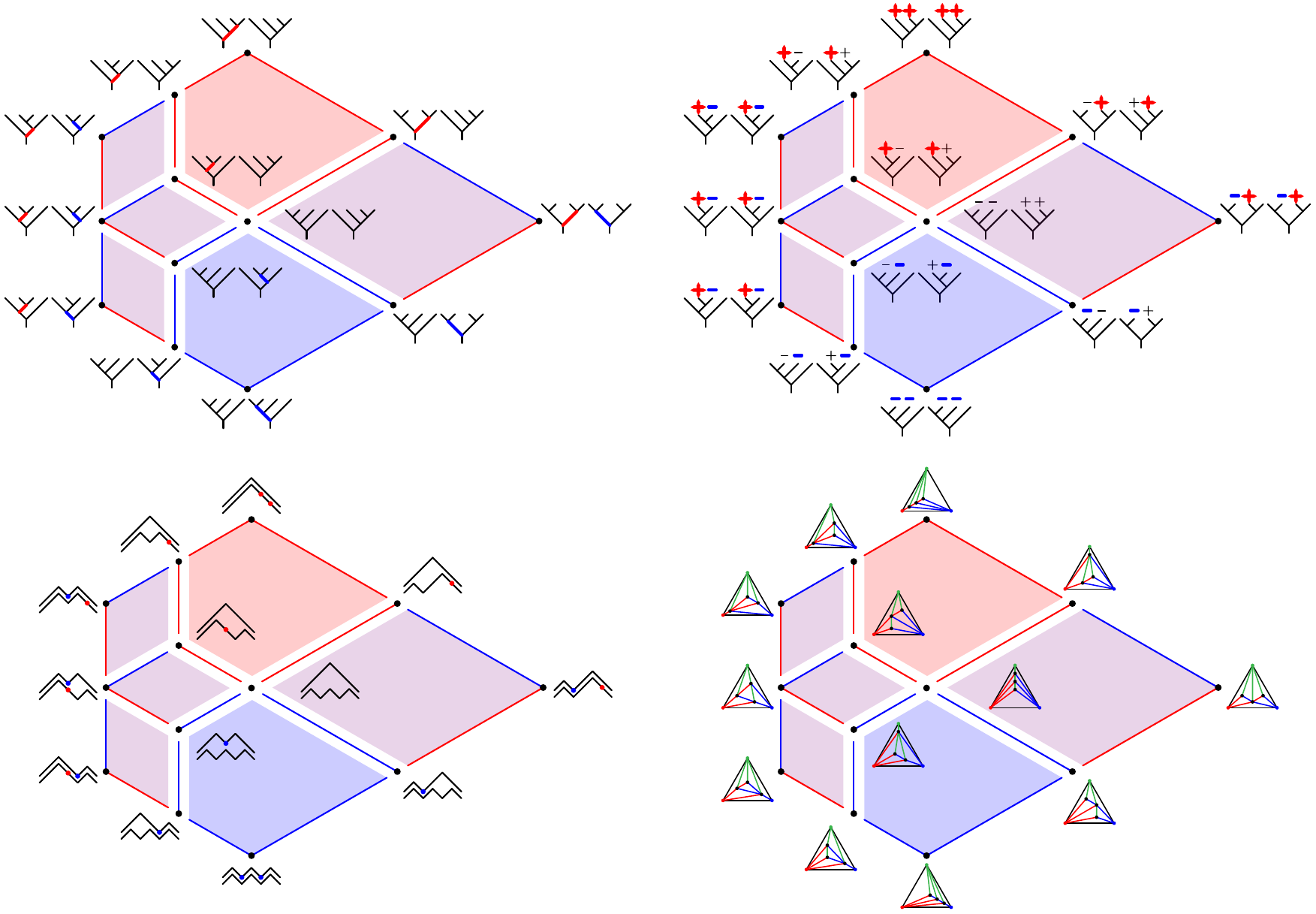}}
	\caption{The decomposition of the cellular diagonal~$\Delta_2$ of \cref{fig:diagonalAssociahedron}, labeled using the equivalent statistics of \cref{fig:ascentsDescentsCanopyDyckPathsSchnyderWoods}.}
	\label{fig:diagonalAssociahedronCanopyDyckPathsSchnyderWoods}
\end{figure}

Using \cref{lem:canopy}, we can transpose \cref{thm:fVectorCanonicalComplex,thm:fVectorDiagonal} in terms of canopy.
We denote by~$\agree(S,T)$ the number of \defn{canopy agreements} between two binary trees~$S$ and~$T$ (\ie of positions where the entries of the canopies of~$S$ and~$T$ agree).

\begin{corollary}
\label{coro:canopy1}
For any~$n,k \in \N$, we have
\[
|\set{S \le T}{\agree(S,T) = k}| = |\set{S \le T}{\des(S) + \asc(T) = k}| = \frac{2}{n(n+1)} \binom{n+1}{k+2} \binom{3n}{k},
\]
where~$S \le T$ are intervals of the Tamari lattice~$\Tam(n)$ on binary trees with~$n$ nodes.
\end{corollary}

\begin{corollary}
\label{coro:canopy2}
For any~$n,k \in \N$, we have
\[
\sum_{S \le T} \binom{\agree(S,T)}{k} = \sum_{S \le T} \binom{\des(S) + \asc(T)}{k} = \frac{2}{(3n+1)(3n+2)} \binom{n-1}{k} \binom{4n+1-k}{n+1},
\]
where the sums range over the intervals~$S \le T$ of the Tamari lattice~$\Tam(n)$ on binary trees with~$n$ nodes.
\end{corollary}

\begin{remark}
For~$k = n-1$ in both \cref{coro:canopy1,coro:canopy2}, we recover that the number of synchronized Tamari intervals (\ie with~$\agree(S,T) = n-1$) is given by
\[
\frac{2}{n(n+1)} \binom{3n}{n-1} = \frac{2}{(n+1)(2n+1)} \binom{3n}{n} = \frac{2}{(3n+1)(3n+2)} \binom{3n+2}{n+1}.
\]
See \cite{PrevilleRatelleViennot,FangPrevilleRatelle}.
\end{remark}

\begin{remark}
Note that the first equalities of \cref{coro:canopy1,coro:canopy2} follow from~\cite[Sect.~5]{Chapoton2}.
The approach of~\cite[Sect.~5]{Chapoton2} is however a bit of a detour as it passes again through generating functions, when the simple observation of \cref{lem:canopy}\,(iii) suffices.
\end{remark}

\subsubsection{Dyck paths}
\label{subsubsec:DyckPaths}

Recall that a \defn{Dyck path} of semilength~$n$ is a path from~$(0,0)$ to~$(2n,0)$ using $n$ up steps~$(1,1)$ (denoted~$U$) and $n$ down steps~$(1,-1)$ (denoted~$D$) and never passing below the horizontal axis.
We denote by~$\pi$ the standard bijection from binary trees to Dyck paths.
Namely, the Dyck path~$\pi(T)$ corresponding to a binary tree~$T$ is obtained by walking clockwise around the contour of~$T$ and marking an~$U$ step when finding a leaf and a $D$ step when walking back an edge~$j \to i$ with~$i < j$.
Note that~$\pi$ transports the rotation on binary trees to the Tamari shift on Dyck paths, which exchanges a $D$ step preceding an $U$ step with the corresponding excursion (meaning the longest subpath which stays above this $U$ step).
See \cref{fig:ascentsDescentsCanopyDyckPathsSchnyderWoods,fig:diagonalAssociahedronCanopyDyckPathsSchnyderWoods} for illustrations.
The following lemma is classical and immediate.

\begin{lemma}
\label{lem:DyckPaths}
The bijection~$\pi$ from binary trees to Dyck path sends:
\begin{itemize}
\item the ascents of~$T$ to the \defn{valleys} of~$\pi(T)$ (a $D$ step followed by an~$U$~step),
\item the descents of~$T$ to the \defn{double falls} of~$\pi(T)$ (two consecutive~$D$ steps),
\item the edges on the left branch of~$T$ to the \defn{contacts} of~$\pi(T)$ (its points on the horizontal axis).
\end{itemize}
\end{lemma}

Using \cref{lem:DyckPaths}, we can transpose \cref{thm:fVectorCanonicalComplex,thm:fVectorDiagonal} in terms of Dyck paths.
We denote by~$\val(P)$ (resp.~$\df(P)$) the number of valleys (resp.~of double falls) of a Dyck path~$P$.

\begin{corollary}
\label{coro:DyckPaths1}
For any~$n,k \in \N$, we have
\[
|\set{P \le Q}{\df(P) + \val(Q) = k}| = \frac{2}{n(n+1)} \binom{n+1}{k+2} \binom{3n}{k},
\]
where~$P \le Q$ are intervals of the Tamari lattice~$\Tam(n)$ on Dyck paths of semilength~$n$.
\end{corollary}

\begin{corollary}
\label{coro:DyckPaths2}
For any~$n,k \in \N$, we have
\[
\sum_{P \le Q} \binom{\df(P) + \val(Q)}{k} = \frac{2}{(3n+1)(3n+2)} \binom{n-1}{k} \binom{4n+1-k}{n+1},
\]
where the sum ranges over the intervals~$P \le Q$ of the Tamari lattice~$\Tam(n)$ on Dyck paths of semilength~$n$.
\end{corollary}

\subsubsection{Triangulations and minimal realizers}
\label{subsubsec:triangulations}

We now consider the bijection of~\cite{BernardiBonichon} from Tamari intervals to rooted triangulations using Schnyder woods.
Schnyder woods were introduced in~\cite{Schnyder} for straightline embedding purposes, and the structure of Schnyder woods was investigated in particular in~\cite{Ossona, Propp, Felsner-latticesOrientations}.
We refer to \cite[Chap.~2]{Felsner} for a nice pedagogical presentation of Schnyder woods and their applications.

Recall that a \defn{planar map}~$M$ is an embedding of a planar graph on the sphere, considered up to continuous deformations.
A \defn{face} of~$M$ is a connected component of the complement of~$M$, and a \defn{corner} is a pair of consecutive edges around a vertex.
A \defn{rooted} map is a map where a root corner is marked.
The face containing this corner is then considered as the \defn{external} face, and the vertices and edges of this external face are the external vertices and edges.
A \defn{triangulation} is a map where all faces have degree~$3$.
Euler formula implies that a rooted triangulation with $n$ internal vertices has~$3n$ internal edges and $2n+1$ internal triangles.

Consider a rooted triangulation~$M$ and denote by~$v_0, v_1, v_2$ the external vertices of~$M$ counterclockwise around the external face, and by~$U$ the internal vertices of~$M$.
A \defn{realizer} (or \defn{Schnyder wood}~\cite{Schnyder}) of~$M$ is an orientation and coloring with colors~$\{0,1,2\}$ of the edges of~$M$ such that
\begin{itemize}
\item for each~$i \in \{0,1,2\}$, the $i$-edges form a tree with vertices~$U \cup \{v_i\}$ oriented towards~$v_i$,
\item counterclockwise around each internal vertex, we see a $0$-source, some $2$-targets, a $1$-source, some $0$-targets, a $2$-source, and some $1$-targets. (Note that some means possibly none.)
\end{itemize}
(An $i$-edge is an edge colored $i$, and an $i$-source or $i$-target is the source or target of and $i$-edge.)
A realizer is \defn{minimal} (resp.~\defn{maximal}) if it contains no clockwise (resp.~counterclockwise) cycle.
It was observed in~\cite{Ossona, Propp, Felsner-latticesOrientations} that the Schnyder woods on a given triangulation~$M$ have the structure of a distributive lattice, where the cover relations correspond to reorientation of certain clockwise cycles.
This has the following immediate consequence.

\begin{theorem}[\cite{Ossona, Propp, Felsner-latticesOrientations}]
Every triangulation has a unique minimal (resp.~maximal) realizer.
\end{theorem}

Consider now a realizer~$(T_0, T_1, T_2)$ of a rooted triangulation~$M$.
Walking clockwise around~$T_0$, we define two Dyck paths~$P$ and~$Q$ as follows:
\begin{itemize}
\item $P$ has an~$U$ (resp.~$D$) step each time we move farther from~$v_0$ (resp.~closer to~$v_0$),
\item $Q$ has an~$U$ step each time we move farther from~$v_0$ (except the first step), and a $D$ step each time we pass a $1$-target.
\end{itemize}
See \cref{fig:ascentsDescentsCanopyDyckPathsSchnyderWoods,fig:diagonalAssociahedronCanopyDyckPathsSchnyderWoods} for illustrations.
This map was defined in~\cite{BernardiBonichon}, where it is proved that it behaves very nicely with respect to three lattice structures on Dyck paths (the Stanley lattice, the Tamari lattice and the Kreweras lattice).
Here, we will use only the connection to the Tamari lattice, but we first make an immediate observation.
We call \defn{intermediate nodes} of a rooted tree~$T$ the nodes which are neither the root, nor the leaves of~$T$.

\begin{lemma}
\label{lem:intermediateNodesRealizers}
Consider the pair~$(P,Q)$ of Dyck paths obtained from a realizer~$(T_0, T_1, T_2)$. Then
\begin{itemize}
\item the double falls of~$P$ correspond to the intermediate nodes of~$T_0$,
\item the valleys of~$Q$ correspond to the intermediate nodes of~$T_1$,
\item the contacts of~$P$ correspond to the corners of edges of~$T_0$ incident to~$v_0$.
\end{itemize}
\end{lemma}

We now restrict to minimal realizers to obtain a bijection between rooted triangulations and Tamari intervals, as described in~\cite{BernardiBonichon}.
We denote by~$\BB(M)$ the pair of Dyck paths~$(P,Q)$ obtained from the minimal realizer of~$M$.

\begin{theorem}[\cite{BernardiBonichon}]
\label{thm:BernardiBonichon}
The map~$\BB$ is a bijection from rooted triangulations with $n$ internal vertices to the intervals of the Tamari lattice on Dyck paths of semilength~$n$.
\end{theorem}

Using \cref{lem:intermediateNodesRealizers,thm:BernardiBonichon}, we can transpose \cref{thm:fVectorCanonicalComplex,thm:fVectorDiagonal} in terms of maps.
For a rooted triangulation~$M$, with minimal realizer~$(T_0, T_1, T_2)$, we denote by~$\intNodes(M)$ the number of intermediate nodes of~$T_0$ plus the number of intermediate nodes of~$T_1$.

\begin{corollary}
\label{coro:realizers1}
For any~$n,k \in \N$, we have
\[
|\set{M}{\intNodes(M) = k}| = \frac{2}{n(n+1)} \binom{n+1}{k+2} \binom{3n}{k},
\]
where the~$M$'s are the rooted triangulations with~$n$ internal vertices.
\end{corollary}

\begin{corollary}
\label{coro:realizers2}
For any~$n,k \in \N$, we have
\[
\sum_{M} \binom{\intNodes(M)}{k} = \frac{2}{(3n+1)(3n+2)} \binom{n-1}{k} \binom{4n+1-k}{n+1},
\]
where the sums range over all rooted triangulations~$M$ with~$n$ internal vertices.
\end{corollary}

\subsection{\cref{thm:fVectorCanonicalComplex} from triangulations}
\label{subsec:triangulations}

We now derive \cref{thm:fVectorCanonicalComplex} from triangulations using the following result of~\cite{FusyHumbert}.
It was obtained via a bijection from planar triangulations endowed with their minimal realizers to planar mobiles.
We state it here in terms of canopies of binary trees.

\begin{theorem}[{\cite[Coro.~2]{FusyHumbert}}]
\label{thm:FusyHumbert}
Let~$f_{i,j,k}$ denote the number of Tamari intervals~$S \le T$ with $i$ positions~$p$ where~$\can(S)_p = \can(T)_p = {-}$, with~$j$ positions~$p$ where~$\can(S)_p = \can(T)_p = {+}$, and with~$k$ positions~$p$ where~$\can(S)_p = {-}$ while~$\can(T)_p = {+}$.
Then the corresponding generating function~$F \eqdef F(u,v,w) \eqdef \sum_{i,j,k} f_{i,j,k} u^i v^j w^k$ is given by
\[
uvF = uU + vV + wUV - \frac{UV}{(1+U)(1+V)},
\]
where the series~$U \eqdef U(u,v,w)$ and~$V \eqdef V(u,v,w)$ satisfy the system
\begin{align*}
U & = (v+wU)(1+U)(1+V)^2 \\
V & = (u+wV)(1+V)(1+U)^2.
\end{align*}
\end{theorem}

\begin{corollary}
\label{coro:FusyHumbert}
The generating function~$A \eqdef A(t,z) \eqdef \sum a_{n,k} t^n z^k$ is given by
\begin{equation}
\label{eq:FusyHumbert1}
t z^2 A = 2tz S + t S^2 - \frac{S^2}{(1+S)^2},
\end{equation}
where the series~$S \eqdef S(t,z)$ satisfies
\begin{equation}
\label{eq:FusyHumbert2}
S = t(z+S)(1+S)^3.
\end{equation}
\end{corollary}

\begin{proof}
By \cref{coro:canopy1}, we have~$A(t,z) = t F(tz, tz, t)$.
Specializing~$u = v = tz$ and~$w = t$ in \cref{thm:FusyHumbert}, we thus obtain the expression for~$A(t,z)$ by observing that the series~$U(tz,tz,t)$ and~$V(tz,tz,t)$ coincide and denoting~$S(t,z) \eqdef U(tz,tz,t) = V(tz,tz,t)$.
\end{proof}

Differentiating \cref{eq:FusyHumbert1} with respect to the variable~$t$, we obtain
\begin{align}
\frac{\partial}{\partial t} (t z^2 A)
& = 2zS + 2tz \frac{\partial S}{\partial t} + S^2 + 2 t S  \frac{\partial S}{\partial t} - \frac{2 S }{(1+S)^2}  \frac{\partial S}{\partial t} + \frac{2 S^2}{(1+S)^3} \frac{\partial S}{\partial t} \nonumber \\
& = 2zS + S^2 + \frac{2}{(1+S)^3} \frac{\partial S}{\partial t} \Big( t(z+S)(1+S)^3 - S(1+S) + S^2 \Big) \nonumber \\
& = 2zS + S^2,
\label{eq:FusyHumbert3}
\end{align}
where the last equality follows from~\cref{eq:FusyHumbert2}.

We obtain by Lagrange inversion in~\cref{eq:FusyHumbert2} that for~$r \ge 1$,
\[
[t^n z^k] S^r = \frac{r}{n} [s^{n-r} z^k] \phi(s)^n,
\]
where~$\phi(s) \eqdef (z+s)(1+s)^3$.
Thus
\[
[t^n z^k] S^r = \frac{r}{n} [s^{n-r} z^k] (z+s)^n (1+s)^{3n} = \frac{r}{n} \binom{n}{k} \binom{3n}{k-r}.
\]
Hence, \cref{eq:FusyHumbert3} implies that
\[
a_{n,k} = [t^n z^k] A = \frac{1}{n+1} [t^n z^{k+2}] \frac{\partial}{\partial t} (t z^2 A) = \frac{1}{n+1} \big( 2 [t^n z^{k+1}] S + [t^n z^{k+2}] S^2 \big)
\]
is given by
\[
\frac{2}{n(n+1)} \left( \binom{n}{k+1} + \binom{n}{k+2} \right) \binom{3n}{k} = \frac{2}{n(n+1)} \binom{n+1}{k+2} \binom{3n}{k}.
\]

\begin{remark}
In fact, the recent direct bijection of~\cite{FangFusyNadeau} between Tamari intervals and blossoming trees enables to obtain \cref{thm:fVectorCanonicalComplex} in an even simpler way.
Details can be found in~\cite{FangFusyNadeau}.
\end{remark}

%%%%%%%%%%%%%%%%%%%%%%%%%%%%%%%%%%%%%%

\section{$q$-analogues}

We conclude this paper with an insightful observation due to a referee, which opens the door to a promising and exciting direction for future research.

Recall that \defn{$q$-analogues} in combinatorics are generalizations of classical numbers and formulas in which integers are replaced by polynomials in a variable $q$, in such a way that the original formulas are recovered when $q \to 1$.
They typically arise by refining a counting problem with an additional statistic (such as inversions), so that the answer becomes a generating function rather than a single number.

A basic building block is the \defn{$q$-integer}~$[n]_q \eqdef 1 + q + \cdots + q^{n-1} = \frac{1-q^n}{1-q}$, and from it the \defn{$q$-factorial}~$[n]_q! \eqdef [1]_q[2]_q\cdots[n]_q$.
The \defn{$q$-binomial coefficients} (or \defn{Gaussian binomial coefficients}) are then defined by
\[
\qbinom{n}{k} \eqdef \frac{[n]_q!}{[k]_q!\,[n-k]_q!} = \frac{(1-q^n)(1-q^{n-1})\cdots(1-q^{\,n-k+1})}{(1-q^k)(1-q^{k-1})\cdots(1-q)}.
\]
As $q \to 1$, one recovers the ordinary integer~$n$, factorial~$n!$, and binomial coefficient $\binom{n}{k}$.
Combinatorially, $\qbinom{n}{k}$ counts $k$-dimensional subspaces of $\mathbb{F}_q^n$ and also serves as a generating function for statistics such as inversions in binary words with $k$ ones and $n-k$ zeros.

\cref{thm:fVectorCanonicalComplex} states that
\[
\sum_{\ell = 0}^{n-1} \binom{n+1}{\ell+2} \binom{3n}{\ell} \bigg/ \binom{n+1}{2} = \binom{4n+1}{n+1} \bigg/ \binom{3n+2}{2}
\]
counts all Tamari intervals~\cite{Chapoton1}.
The following $q$-analogue was conjectured by an anonymous referee.

\begin{proposition}
\label{prop:referee1}
For any~$n,k \in \N$, 
\[
\sum_{\ell = 0}^{n-1} q^{\ell(\ell+2)} \qbinom{n+1}{\ell+2} \qbinom{3n}{\ell} \bigg/ \qbinom{n+1}{2} = \qbinom{4n+1}{n+1} \bigg/ \qbinom{3n+2}{2}.
\]
\end{proposition}

%sage: var('q')
%sage: for n in range(10):
%....:     print(bool(add(q^(k*(k+2)) * q_binomial(n+1, k+2) * q_binomial(3*n, k) / q_binomial(n+1, 2) for k in range(n)) == q_binomial(4*n+1, n+1) / q_binomial(3*n+2, 2)))

Further experiments lead us to extend \cref{prop:referee1}.
\cref{thm:fVectorDiagonal} states that
\[
\sum_{\ell = k}^{n-1} \binom{n+1}{\ell+2} \binom{3n}{\ell} \binom{\ell}{k} \bigg/ \binom{n+1}{2} = \binom{n-1}{k} \binom{4n+1-k}{n+1} \bigg/ \binom{3n+2}{2}.
\]
We observed the following $q$-analogue.

\begin{proposition}
\label{prop:referee2}
For any~$n,k \in \N$, 
\[
\sum_{\ell = k}^{n-1} q^{(\ell-k)(\ell+2)} \qbinom{n+1}{\ell+2} \qbinom{3n}{\ell} \qbinom{\ell}{k} \bigg/ \qbinom{n+1}{2} = \qbinom{n-1}{k} \qbinom{4n+1-k}{n+1} \bigg/ \qbinom{3n+2}{2}.
\]
\end{proposition}

%sage: var('q')
%sage: for n in range(10):
%....:     for k in range(n):
%....:         X = add(q^((l-k)*(l+2)) * q_binomial(n+1, l+2) * q_binomial(3*n, l) * q_binomial(l, k) / q_binomial(n+1, 2) for l in range(k,n))
%....:         Y = q_binomial(n-1, k) * q_binomial(4*n+1-k, n+1) / q_binomial(3*n+2, 2)
%....:         print(n, k, bool(X == Y))

\cref{prop:referee1,prop:referee2} can be proved for instance using computer algebra tools, such as creative telescoping~\cite{Zeilberger-1991-MCT,Zeilberger-1990-FAP,PetkovsekWilfZeilberger-1996-AB}.
We leave their combinatorial interpretations for future research.

%%%%%%%%%%%%%%%%%%%%%%%%%%%%%%%%%%%%%%

\section*{Aknowledgements}

VP thanks all participants of the ``2023 Barcelona Workshop: Homotopy theory meets polyhedral combinatorics'' (Mónica Blanco, Luis Crespo, Guillaume Laplante-Anfossi, Arnau Padrol, Eva Philippe, Julian Pfeifle, Daria Poliakova, Francisco Santos and Andy Tonks) where the question to understand the $f$-vector of the cellular diagonal of the associahedron was raised. VP is particularly grateful to Guillaume Laplante-Anfossi for various discussions on the cellular diagonal of the associahedron and for an amazing number of suggestions on the presentation of the paper, and to Francisco Santos for suggesting the rightmost picture of \cref{fig:diagonalAssociahedron}. VP thanks \'Eric Fusy for suggesting the bijective approach of \cref{subsec:triangulations} and answering technical questions on this approach.
We also thank Frédéric Chapoton, Florent Hivert, Pierre Lairez, and Gilles Schaeffer for interesting discussions and inputs on this paper.
Finally, we are grateful to anonymous referees for various comments and suggestions on previous versions of this paper, and in particular to the referee of Elect.~J.~Comb. for conjecturing \cref{prop:referee1}.

%%%%%%%%%%%%%%%%%%%%%%%%%%%%%%%%%%%%%%

\bibliographystyle{alpha}
\bibliography{TamariIntervals}
\label{sec:biblio}

\end{document}